\documentclass{amsart}

\usepackage{hyperref}
\usepackage{amsmath,amsthm,amssymb}

\newtheorem{thm}{Theorem}
\newtheorem{prop}[thm]{Proposition}
\newtheorem{lem}[thm]{Lemma}
\newtheorem{cor}[thm]{Corollary}
\newtheorem{clm}[thm]{Claim}

\theoremstyle{definition}
\newtheorem{dfn}[thm]{Definition}
\newtheorem{rem}[thm]{Remark}
\newtheorem{ex}[thm]{Example}

\theoremstyle{remark}
\newtheorem*{org}{Organization}
\newtheorem*{ack}{Acknowledgment}

\numberwithin{thm}{section}
\numberwithin{equation}{section}

\DeclareMathOperator{\dist}{dist}
\DeclareMathOperator{\diam}{diam}
\DeclareMathOperator{\vol}{vol}
\DeclareMathOperator{\str}{str}

\begin{document}

\title[Regular points of extremal subsets]{Regular points of extremal subsets in Alexandrov spaces}
\author[T. Fujioka]{Tadashi Fujioka}
\address{Department of Mathematics, Kyoto University, Kitashirakawa, Kyoto 606-8502, Japan}
\email{\href{mailto:tfujioka@math.kyoto-u.ac.jp}{tfujioka@math.kyoto-u.ac.jp}}
\date{\today}
\thanks{Supported in part by JSPS KAKENHI Grant Number 15H05739}

\begin{abstract}
We define regular points of an extremal subset in an Alexandrov space and study their basic properties.
We show that a neighborhood of a regular point in an extremal subset is almost isometric to an open subset in Euclidean space and that the set of regular points in an extremal subset has full measure and is dense in it.
These results actually hold for strained points in an extremal subset.
Applications include the volume convergence of extremal subsets under a noncollapsing convergence of Alexandrov spaces, and the existence of a cone fibration structure of a metric neighborhood of the regular part of an extremal subset.
In an appendix, a deformation retraction of a metric neighborhood of a general extremal subset is constructed.
\end{abstract}

\maketitle

\section{Introduction}\label{sec:intro}

Alexandrov spaces are a generalization of Riemannian manifolds having a lower sectional curvature bound.
They naturally arise as Gromov-Hausdorff limits of Riemannian manifolds with sectional curvatures uniformly bounded below.
The fundamental theory of Alexandrov spaces was developed by Burago, Gromov and Perelman \cite{BGP}.
In general, Alexandrov spaces contain singular points.
Extremal subsets are a kind of singular point sets in Alexandrov spaces introduced by Perelman and Petrunin \cite{PP1}, which are closely related to the stratification of Alexandrov spaces.
The boundary of an Alexandrov space is a typical example of an extremal subset.
It is known that any Alexandrov space admits a stratification into topological manifolds such that the closures of its strata are all possible primitive extremal subsets of that space.
The aim of this paper is to define and study regular points of such strata.

Let us first recall the basic results on regular points of an Alexandrov space.
Let $M$ be an $n$-dimensional Alexandrov space (note that the Hausdorff dimension coincides with the topological dimension).
A regular point of $M$ is a point whose tangent cone is isometric to $\mathbb R^n$.
Then, a neighborhood of a regular point is almost isometric to an open subset of $\mathbb R^n$ (\cite[9.4]{BGP}).
Furthermore, the Hausdorff dimension of the set of non-regular points is no more than $n-1$ (\cite[10.6]{BGP}, \cite[Theorem A]{OS}), and in particular, the set of regular points is dense in $M$ (\cite[6.7]{BGP}).

We prove similar statements for each extremal subset in $M$.
For the sake of dimensional homogeneity, we only consider primitive extremal subsets (indeed, every extremal subset is uniquely represented as a union of primitive extremal subsets with nonempty relative interior).
Let $E$ be a primitive extremal subset of $M$.
The dimension of $E$, denoted by $m$, is naturally defined by the Hausdorff dimension (see Theorem \ref{thm:main}(1) below).
As in the case of $M$, a \textit{regular point} of $E$ is defined as a point whose tangent cone of $E$ is isometric to $\mathbb R^m$.
Here, the tangent cone of $E$ is equipped with the \textit{extrinsic metric}, that is, the restriction of the metric of the ambient space (on the other hand, the \textit{intrinsic metric} means the one induced by the length structure).
The following is our main result:

\begin{thm}\label{thm:main}
Let $E$ be a primitive extremal subset of an Alexandrov space $M$.
\begin{enumerate}
\item The Hausdorff dimension of $E$ coincides with the topological dimension.
Moreover, it coincides with the maximal dimension of open subsets of $E$ homeomorphic to Euclidean balls.
Let us denote it by $m$.
\item Let $p$ be a regular point of $E$.
Then, for any $\varepsilon>0$, there exists a neighborhood of $p$ in $E$ that is $\varepsilon$-almost isometric to an open subset in $\mathbb R^m$ with respect to both intrinsic and extrinsic metrics of $E$.
\item Let $E_0$ be the set of regular points in $E$.
Then, the Hausdorff dimension of the complement $E\setminus E_0$ is no more than $m-1$.
In particular, $E_0$ is dense in $E$.
\end{enumerate}
\end{thm}

Note that the intrinsic and extrinsic metrics of $E$ are known to be locally bi-Lipschitz equivalent.
The above theorem implies that both metrics are almost isometric in a neighborhood of almost every point in $E$.
Recently, Li and Naber \cite{LiN} obtained a much better control on singular sets of Alexandrov spaces than (3).
In particular, $E\setminus E_0$ is $(m-1)$-rectifiable (\cite[1.6]{LiN}).
In the case of a boundary, (2) was known in \cite[12.9.1]{BGP} in a stronger form, and the denseness of regular points was shown in \cite[6.6]{HS}.

Regular points are closely related to strained points.
Recall that a point $p\in M$ is called a $(k,\delta)$-strained point if there exists a collection $\{(a_i,b_i)\}_{i=1}^k$ of pairs of points in $M$ such that
\begin{gather*}
\tilde\angle a_ipb_i>\pi-\delta,\\
\tilde\angle a_ipa_j,\tilde\angle a_ipb_j,\tilde\angle b_ipb_j>\frac\pi2-\delta
\end{gather*}
for all $1\le i\neq j\le k$.
Then, the splitting theorem implies that a point in $M$ is regular if and only if it is $(n,\delta)$-strained for any $\delta>0$.
In particular, the properties of regular points of $M$ mentioned above also hold for $(n,\delta)$-strained points.
Similarly, the following holds:

\begin{thm}\label{thm:main'}
Let $E$ be an $m$-dimensional primitive extremal subset of an Alexandrov space $M$.
Then, a point in $E$ is regular if and only if it is $(m,\delta)$-strained for any $\delta>0$.
In particular, (2) and (3) of Theorem \ref{thm:main} hold for $(m,\delta)$-strained points instead of regular points (provided $\delta\ll\varepsilon$ in (2)).
\end{thm}

Furthermore, the complement of the set of $(m,\delta)$-strained points in $E$ has locally finite $(m-1)$-dimensional Hausdorff measure, which is a direct consequence of \cite[10.6]{BGP}.
It should be noted that there is a much stronger result in \cite[1.4]{LiN}.
See Theorem \ref{thm:str} and Remark \ref{rem:str}.

In the proof of Theorem \ref{thm:main}, the generalized Lieberman theorem \cite[5.3]{PP1} plays an important role, which is a kind of totally geodesic property of extremal subsets.
The almost isometry in Theorem \ref{thm:main}(2) is given by the distance map $f=(|a_1\cdot|,\dots,|a_m\cdot|)$ from an $(m,\delta)$-strainer $\{(a_i,b_i)\}_{i=1}^m$ at $p$.

As applications, we obtain the following two theorems for extremal subsets.
One is the volume convergence of extremal subsets.
Recall that the volume of Alexandrov spaces is continuous under the Gromov-Hausdorff convergence (\cite[10.8]{BGP}, \cite[3.5]{S}, \cite[0.6]{Y1}) and that the proof is based on the properties of regular (strained) points.
Let $\vol_m$ denote the $m$-dimensional Hausdorff measure.
Note that by Theorem \ref{thm:main}, the Hausdorff measure on an extremal subset is independent of which metric (intrinsic or extrinsic) is considered.

\begin{thm}\label{thm:vol}
Let $M_i\overset{\mathrm{GH}}\longrightarrow M$ be a noncollapsing sequence of Alexandrov spaces and suppose that $m$-dimensional extremal subsets $E_i$ of $M_i$ converge to an $m$-dimensional compact extremal subset $E$ of $M$.
Then, $\lim_{i\to\infty}\vol_mE_i=\vol_mE$.

Moreover, $(E_i,\vol_m)$ converges to $(E,\vol_m)$ in the measured Gromov-Hausdorff sense.
The same is true for the pointed Gromov-Hausdorff convergence if $E$ is noncompact.

\end{thm}

If $M_i$ are $n$-dimensional Alexandrov spaces with curvature $\ge\kappa$, diameter $\le D$ and volume $\ge v>0$, and $E_i$ are $m$-dimensional extremal subsets of $M_i$, then there exists a convergent subsequence as above (in particular, $\dim E=m$; see \cite[9.13]{K} for instance).
If $M_i$ collapses, then the theorem does not hold in general, even if $E_i$ does not collapse.
For example, let $M$ be an Alexandrov space with nonempty boundary and consider $M_\varepsilon:=M\times [0,\varepsilon]$ and $E_\varepsilon:=\partial M_\varepsilon=\partial M\times\{0,\varepsilon\}\cup M\times[0,\varepsilon]$, where $\varepsilon\to 0$.
Then, $\vol E_\varepsilon$ does not converge to $\vol M$, while $E_\varepsilon$ converges to $M$.

The other result is about the relation between an extremal subset and its neighborhood in the ambient Alexandrov space.
It is known that the boundary of an Alexandrov space has a collar neighborhood (\cite[5.10]{Y2}).
The following result is analogous to Yamaguchi's theorem \cite[0.2]{Y1} (see also \cite{F2}) stating that if a sequence of Alexandrov spaces collapses to a space with only weak singularities, then the collapsing spaces admit almost Lipschitz submersions onto the limit space.

\begin{thm}\label{thm:nhd}
Let $E$ be an $m$-dimensional compact primitive extremal subset of an Alexandrov space $M$.
Suppose that every point in $E$ is regular (or more generally, $(m,\delta)$-strained, where $\delta$ depends on $M$).
Then, for sufficiently small $\varepsilon>0$, the $\varepsilon$-neighborhood $U_\varepsilon(E)$ of $E$ is a fiber bundle over $E$ with conical fiber.
\end{thm}

In particular, $E$ is a deformation retract of $U_\varepsilon(E)$.
This is actually the case for any compact extremal subset.
The proof is given in an appendix since it is not related with the regular point theory.

\begin{thm}\label{thm:def}
Let $E$ be a compact extremal subset of an Alexandrov space $M$.
Then, for any sufficiently small $\varepsilon>0$, $U_\varepsilon(E)$ admits a Lipschitz deformation retraction onto $E$.
\end{thm}

\begin{org}
The organization of this paper is as follows:
In \S\ref{sec:pre}, we review the basics of Alexandrov spaces and extremal subsets, including the notions of strainers, gradient curves and quasigeodesics.
In \S\ref{sec:bas}, we first show that the distance map from a strainer is open on an extremal subset and that the Hausdorff dimension of an extremal subset coincides with its strainer number.
We then define regular points of a primitive extremal subset and prove Theorems \ref{thm:main}, \ref{thm:main'} and some corollaries.
In \S\ref{sec:appl}, we prove Theorem \ref{thm:vol} and Theorem \ref{thm:nhd}.
In Appendix \ref{sec:A}, we prove Theorem \ref{thm:def}.
\end{org}

\begin{ack}
The author would like to thank Prof.\ Takao Yamaguchi for his advice and encouragement.
He is also grateful to the referee for the careful reading and useful comments.
\end{ack}

\section{Preliminaries}\label{sec:pre}

\subsection{Alexandrov spaces}\label{sec:alex}

We refer to \cite{BGP} and \cite{BBI} for the the basics of Alexandrov spaces.

For three points $p$, $q$ and $r$ in a metric space, consider a triangle in the $\kappa$-plane, the simply-connected complete surface of constant curvature $\kappa$, with sidelengths $|pq|$, $|qr|$ and $|rp|$.
We denote by $\tilde\angle qpr$ the angle opposite to the side $|qr|$ and call it a \textit{comparison angle} at $p$.
A complete metric space $M$ is called an \textit{Alexandrov space} with curvature $\ge\kappa$ if it satisfies the following conditions:
\begin{enumerate}
\item any two points in $M$ can be connected by a shortest path;
\item every point has a neighborhood $U$ such that for any four points $p,q,r,s\in U$, we have
\[\tilde\angle qpr+\tilde\angle rps+\tilde\angle spq\le2\pi.\]
\end{enumerate}

In this paper, we only deal with finite-dimensional Alexandrov spaces in the sense of Hausdorff dimension.
The Hausdorff dimension of a finite-dimensional Alexandrov space coincides with the topological dimension, and in particular, it is an integer.
From now, $M$ denotes an $n$-dimensional Alexandrov space with curvature $\ge\kappa$.
Furthermore, $\Sigma$ often denotes a space of curvature $\ge1$.

We assume that every shortest path is parametrized by arclength.
For two shortest paths $\gamma$, $\sigma$ in $M$ starting at $p$, the comparison angle $\tilde\angle\gamma(t)p\sigma(s)$ is nonincreasing in both $t$ and $s$.
Hence, the \textit{angle} $\angle(\gamma,\sigma):=\lim_{t,s\to0}\tilde\angle\gamma(t)p\sigma(s)$ always exists.
The angle becomes a metric on the equivalence classes of shortest paths starting at $p$.
The \textit{space of directions} $\Sigma_p$ at $p$ is the completion of this space.
The \textit{tangent cone} $T_p$ at $p$ is the Euclidean cone over $\Sigma_p$.
Then, $\Sigma_p$ (resp.\ $T_p$) is an $(n-1)$-dimensional compact (resp.\ $n$-dimensional noncompact) Alexandrov space of curvature $\ge1$ (resp.\ $\ge0$).
More generally, a space $\Sigma$ has curvature $\ge1$ if and only if its Euclidean cone $K(\Sigma)$ has curvature $\ge0$.
Note that the rescaled space $(\lambda M,p)$ converges to $(T_p,o)$ as $\lambda\to\infty$ in the pointed Gromov-Hausdorff topology, where $o$ denotes the vertex of the cone.

We will use the following notation:
For two points $p$, $q$ in $M$, $q'_p$ denotes one of the directions of shortest paths from $p$ to $q$.
Similarly, for a closed subset $A$ in $M$, $A'_p$ denotes the set of all directions of shortest paths from $p$ to $A$.
Moreover, we sometimes denote by $Q'_p$ the set of all directions of shortest paths from $p$ to $q$ by regarding $Q$ as a closed subset $\{q\}$.

It is well-known that the class of Alexandrov spaces with dimension $\le n$, curvature $\ge\kappa$ and diameter $\le D$ is compact with respect to the Gromov-Hausdorff distance.
Similarly, the class of pointed Alexandrov spaces with dimension $\le n$ and curvature $\ge\kappa$ is compact with respect to the pointed Gromov-Hausdorff topology.

\subsubsection{Strainers}\label{sec:str}

We refer to \cite{BGP} and \cite[Chap.\ 10]{BBI} for more details on strainers.

Throughout this paper, we assume that a positive number $\delta$ is less than some constant depending only on the dimension and the lower curvature bound of Alexandrov spaces, such as $n$ and $\kappa$.
We also denote by $\varkappa(\delta)$ various positive functions depending only on $n$ and $\kappa$ such that $\varkappa(\delta)\to0$ as $\delta\to0$.

\begin{dfn}\label{dfn:str}
Let $M$ be an Alexandrov space and $p\in M$.
A collection  $\{(a_i,b_i)\}_{i=1}^k$ of pairs of points in $M$ is called a \textit{$(k,\delta)$-strainer} at $p$ if it satisfies
\begin{gather*}
\tilde\angle a_ipb_i>\pi-\delta,\\
\tilde\angle a_ipa_j,\tilde\angle a_ipb_j,\tilde\angle b_ipb_j>\frac\pi2-\delta
\end{gather*}
for all $1\le i\neq j\le k$.
Such a point $p$ is called a \textit{$(k,\delta)$-strained point}.
The number $\min_{1\le i\le k}\{|pa_i|,|pb_i|\}$ is called the \textit{length} of the strainer $\{(a_i,b_i)\}_{i=1}^k$.

Furthermore, let $\Sigma$ be a space of curvature $\ge1$.
A collection $\{(A_i,B_i)\}_{i=1}^k$ of pairs of compact subsets in $\Sigma$ is called a \textit{$(k,\delta)$-strainer} of $\Sigma$ if it satisfies
\begin{gather*}
|A_iB_i|>\pi-\delta,\\
|A_iA_j|,|A_iB_j|,|B_iB_j|>\frac\pi2-\delta
\end{gather*}
for all $1\le i\neq j\le k$.
\end{dfn}

Clearly, $p\in M$ is $(k,\delta)$-strained if and only if $\Sigma_p$ has a $(k,\delta)$-strainer.

Let $X$ be a subset of $M$.
The \textit{$\delta$-strainer number} of $X$, denoted by $\delta\mathchar`-\str(X)$, is the supremum of numbers $k$ such that there exists a $(k,\delta)$-strained point in $X$.
Moreover, the \textit{local $\delta$-strainer number} of $X$ at $p\in X$, denoted by $\delta\mathchar`-\str_p(X)$, is the supremum of numbers $k$ such that every neighborhood of $p$ in $X$ contains a $(k,\delta)$-strained point.
It is known that the local $\delta$-strainer number of $M$ at any point is $n$ for any small $\delta>0$.

Let $M(k,\delta)$ be the set of $(k,\delta)$-strained points in $M$.
Set $M(k):=\bigcap_{\delta>0}M(k,\delta)$.
We denote by $\dim_{\mathrm H}$ the Hausdorff dimension.
Then, the following holds:

\begin{thm}[{\cite[10.6]{BGP}}]\label{thm:str}
For any $\delta>0$, the complement $M\setminus M(k,\delta)$ has locally finite $(k-1)$-dimensional Hausdorff measure.
In particular,
\begin{gather*}
\dim_{\mathrm H}\left(M\setminus M(k,\delta)\right)\le k-1,\\
\dim_{\mathrm H}\left(M\setminus M(k)\right)\le k-1.
\end{gather*}
\end{thm}

\begin{rem}\label{rem:str}
Recently, Li and Naber \cite{LiN} obtained a much stronger result on singular sets of Alexandrov spaces.
In particular, they gave an upper bound on the $(k-1)$-dimensional Hausdorff measure of $M\setminus M(k,\delta)$ and showed that $M\setminus M(k)$ is $(k-1)$-rectifiable.
See \cite{LiN} for more details.
\end{rem}

Let $\Sigma$ be a space of curvature $\ge1$.
If $(\xi,\eta)$ is an opposite pair in $\Sigma$, i.e.\ $|\xi\eta|=\pi$, then $\Sigma$ is isometric to the spherical suspension $S(Z)$ over a space $Z$ of curvature $\ge1$, where $Z=\{\zeta\in\Sigma\mid|\xi\zeta|=|\eta\zeta|=\pi/2\}$ (see \cite[10.4.3]{BBI}).
Therefore, if there exists a collection $\{(\xi_i,\eta_i)\}_{i=1}^k$ of pairs of points in $\Sigma$ such that 
\[|\xi_i\eta_i|=\pi,\quad|\xi_i\xi_j|=|\xi_i\eta_j|=|\eta_i\eta_j|=\pi/2\]
for all $1\le i\neq j\le k$, then $\Sigma$ is isometric to the $k$-fold spherical suspension
\[S^k(\hat\Sigma)=\overbrace{S(S(\dots S}^k(\hat\Sigma)\dots))\]
over some space $\hat\Sigma$ of curvature $\ge1$.
We refer to such a collection $\{(\xi_i,\eta_i)\}_{i=1}^k$ as an \textit{orthogonal $k$-frame} of $\Sigma$.
Note that $S^k(\hat\Sigma)$ is isometric to the spherical join of the standard unit sphere $\mathbb S^{k-1}$ and $\hat\Sigma$.
We often regard $\mathbb S^{k-1}$ as being isometrically embedded in $S^k(\hat\Sigma)$.
Note that $p\in M$ belongs to $M(k)$ if and only if $\Sigma_p$ has an orthogonal $k$-frame.
In particular, $M(n)$ coincides with the set of regular points in $M$.

\subsection{Extremal subsets}\label{sec:ext}

We refer to \cite{PP1} and \cite[\S4]{Pet2} for the basics of extremal subsets.
See also \cite{Per2} for the stratification of Alexandrov spaces.

Let $M$ be an Alexandrov space.
We denote by $\dist_q$ the distance function $|q\cdot|$ from $q\in M$.
Recall that the first variation formula states $d_p\dist_q=-\cos|Q'_p\cdot|$ on $\Sigma_p$ for any $p\in M\setminus\{q\}$, where $Q'_p$ denotes the set of all directions of shortest paths from $p$ to $q$.

\begin{dfn}\label{dfn:ext}
A closed subset $E$ of $M$ is said to be \textit{extremal} if it satisfies the following condition:
if $p\in E$ is a local minimum point of $\dist_q|_E$ for $q\notin E$, then it is a critical point of $\dist_q$, i.e.\ $d_p\dist_q\le 0$.
Note that the empty set and $M$ itself are regarded as extremal subsets of $M$.

Furthermore, let $\Sigma$ be a space of curvature $\ge1$.
A closed subset $F$ of $\Sigma$ is said to be \textit{extremal} if it satisfies the above condition, and in addition, $\diam\Sigma\le\pi/2$ when $F$ is empty and $\Sigma\subset\bar B(F,\pi/2)$ when $F$ is one point.
\end{dfn}

The dimension of an extremal subset is defined by the Hausdorff dimension.
We will see that it coincides with the strainer number and the topological dimension.
Note that the local dimension is not constant in general.

\begin{ex}\label{ex:ext}
\begin{enumerate}
\item A point whose space of directions has diameter not greater than $\pi/2$ is an extremal subset of dimension $0$.
\item The boundary of an Alexandrov space is an extremal subset of codimension $1$.
The set of all topological singularities (i.e.\ non-manifold points) is also an extremal subset.
\item Let a compact group $G$ act on $M$ by isometries.
Then, the quotient $M/G$ is an Alexandrov space with the same lower curvature bound and the natural projection of the fixed point set of $G$ is an extremal subset of $M/G$.
\end{enumerate}
\end{ex}

Let $E\subset M$ be an extremal subset and $p\in E$.
The \textit{space of directions} of $E$ at $p$, denoted by $\Sigma_p E$, is defined as the subset of $\Sigma_p$ consisting of all limit directions $\lim_{i\to\infty}(p_i)'_p$, where $p_i\in E$ converges to $p$.
Then, $\Sigma_pE$ is an extremal subset of $\Sigma_p$ (in the latter sense of Definition \ref{dfn:ext}).
The \textit{tangent cone} of $E$ at $p$, denoted by $T_pE$, is defined as the subcone $K(\Sigma_pE)$ of $T_p$.
Alternatively, $T_pE$ is defined as the limit of $\lambda E$ under the convergence $(\lambda M,p)\to(T_p,o)$, where $\lambda\to\infty$.
Then, $T_pE$ is also an extremal subset of $T_p$.
More generally, $F$ is an extremal subset of a space $\Sigma$ of curvature $\ge1$ if and only if $K(F)$ is an extremal subset of $K(\Sigma)$.
Furthermore, the limit of extremal subsets under the Gromov-Hausdorff convergence of Alexandrov spaces is an extremal subset of the limit space.

Let $d$ denote the metric of $M$.
An extremal subset $E\subset M$ has two metrics.
One is the \textit{extrinsic metric}, that is, the restriction of $d$.
The other is the \textit{intrinsic metric} $d_E$, that is, the infimum of the lengths of curves on $E$ with respect to $d$.
Clearly $d\le d_E$.
Furthermore, $d_E$ is locally bounded above by a constant multiple of $d$.
More precisely, if $\vol_n B(p,D)\ge v>0$ for $p\in E$, then there exists a constant $C=C(n,\kappa,D,v)$ such that $d_E\le Cd$ on $B(p,D)\cap E$.

Now, we explain the stratification of $M$ by its primitive extremal subsets.
It is known that the union, intersection, and closure of the difference of two extremal subsets are also extremal.
An extremal subset is said to be \textit{primitive} if it contains no proper extremal subsets with nonempty relative interior.
Then, any extremal subset can be uniquely represented as a union of primitive extremal subsets with nonempty relative interior.
Let $E$ be a primitive extremal subset of $M$.
The \textit{main part} $\mathring E$ of $E$ is defined as $E$ with all proper extremal subsets removed.
Then, $\mathring E$ is open and dense in $E$.
Clearly, the main parts of primitive extremal subsets of $M$ form a disjoint covering of $M$.
Furthermore,

\begin{thm}[{\cite[3.8]{PP1}}]
$\mathring E$ is a topological manifold.
In particular, the main parts of primitive extremal subsets of $M$ gives a stratification of $M$.
\end{thm}

\subsubsection{Gradient curves}\label{sec:grad}

We refer to \cite{Pet2} and \cite[\S3]{PP2} for more details on gradient curves.

Extremal subsets have an alternative definition in terms of gradient curves.
Let $p\in M$.
Recall that $d_x\dist_p=-\cos|P'_x\cdot|$ on $\Sigma_x$ for any $x\in M\setminus\{p\}$, where $P'_x$ denotes the set of all directions of shortest paths from $x$ to $p$.
The \textit{gradient} $\nabla_x\dist_p\in T_x$ of $\dist_p$ at $x$ is defined by
\begin{equation*}
\nabla_x\dist_p:=
\begin{cases}
\hfil o & \text{if}\ \max_{\Sigma_x}d_x\dist_p\le 0,\\
d_x\dist_p(\xi)\cdot\xi & \text{if}\ \max_{\Sigma_x}d_x\dist_p>0,
\end{cases}
\end{equation*}
where $\xi\in\Sigma_x$ is the unique maximum point of $d_x\dist_p$ on $\Sigma_x$.
A curve $\alpha$ in $M$ satisfying $\alpha^+(t)=\nabla_{\alpha(t)}\dist_p$ is called a \textit{gradient curve} of $\dist_p$.
It is known that there exists a unique maximal gradient curve of $\dist_p$ starting at $x\neq p$.
The following is an equivalent condition for extremality:

\begin{thm}[{\cite[4.1.2]{Pet2}}]\label{thm:grad}
A subset $E$ of $M$ is extremal if and only if for any $p\in M$ and $x\in E$ with $p\neq x$, the gradient curve of $\dist_p$ starting at $x$ remains in $E$.
\end{thm}

\subsubsection{Quasigeodesics}\label{sec:qg}

We refer to \cite{PP2} and \cite[\S5]{Pet2} for more details on quasigeodesics.

For a $1$-Lipschitz curve $c(t)$ and a point $p$ in $M$, consider a triangle on the $\kappa$-plane with sidelengths $|pc(t_1)|$, $|pc(t_2)|$ and $|t_1-t_2|$.
We denote by $\tilde\angle pc(t_1)\!\smile\!c(t_2)$ the angle opposite to the side $|pc(t_2)|$.
If such a triangle does not exist, then we define $\tilde\angle:=0$.

A curve $\gamma(t)$ in $M$ is called a \textit{quasigeodesic} if it is parametrized by arclength and satisfies the following monotonicity:
for any $p\in M$ and $t\in\mathbb R$, $\tilde\angle p\gamma(t)\!\smile\!\gamma(t+\tau)$ is nonincreasing in $0<\tau<\pi/\sqrt\kappa$ (where $\pi/\sqrt\kappa:=\infty$ for $\kappa\le0$).

The following is a kind of totally geodesic property of extremal subsets, which plays a key role in the proof of Theorem \ref{thm:main}.

\begin{thm}[Generalized Lieberman theorem {\cite[5.3]{PP1}}, {\cite[1.1]{Pet1}}, {\cite[2.3.1]{Pet2}}]\label{thm:qg}
Any shortest path with respect to the intrinsic metric of an extremal subset is a quasigeodesic in the ambient Alexandrov space.
\end{thm}

Let $E\subset M$ be an extremal subset and $p,q\in E$.
Let $\gamma$ be a shortest path from $p$ to $q$ with respect to the intrinsic metric of $E$.
Since a quasigeodesic has a unique right tangent vector, the direction $\gamma^+(0)\in\Sigma_pE$ is uniquely defined.
We denote by $q^\circ_p$ one of such directions.
Furthermore, we define $\tilde\angle xp\!\smile\!q:=\tilde\angle x\gamma(0)\!\smile\!\gamma(|pq|_E)$ for $x\in M$ (in other words, the angle opposite to the side $|xq|$ of a triangle on the $\kappa$-plane with sidelengths $|xp|$, $|pq|_E$ and $|xq|$).
Then, the above monotonicity implies $\tilde\angle xp\!\smile\!q\le\angle(x'_p,q^\circ_p)$.

\section{Basic properties}\label{sec:bas}

In this section, we study the relation between extremal subsets and strainers, define regular points of extremal subsets and prove Theorems \ref{thm:main} and \ref{thm:main'}.

Recall that $\delta$ denotes a positive number less than some constant depending only on $n$ and $\kappa$.
Furthermore, $\varkappa(\delta)$ denotes various positive functions depending only on $n$ and $\kappa$ such that $\varkappa(\delta)\to0$ as $\delta\to0$.
We always assume that a lower bound $\ell$ of the lengths of strainers is not greater than $1$.
Indeed, all our arguments using strainers are local.

Let $\varepsilon$ be a small positive number.
Let $X$ and $Y$ be metric spaces.
A map $f:X\to Y$ is called an \textit{$\varepsilon$-almost isometry} if it is surjective and $||f(x)f(y)|/|xy|-1|<\varepsilon$ for any two distinct points $x,y\in X$.
Let $U$ be an open subset of $X$.
A map $f:U\to Y$ is called a \textit{$(1-\varepsilon)$-open map} if for any $x\in U$ and $v\in Y$ such that $\bar B(x,(1-\varepsilon)^{-1}|f(x)v|)\subset U$, there exists $y\in U$ such that $f(y)=v$ and $(1-\varepsilon)|xy|\le|f(x)v|$.

We first show that the distance map from a strainer is open on an extremal subset (Proposition \ref{prop:open}).
The following lemma gives a sufficient condition for $(1-\varepsilon)$-openness.
This is a special case of \cite[1.2]{L}.

\begin{lem}\label{lem:open}
Let $X$ be a proper metric space (i.e.\ every closed ball is compact) and let $f:U\to\mathbb R^k$ be a continuous map on an open subset $U$ of $X$.
Suppose that for any $p\in U$ and $\xi\in \mathbb S^{k-1}\subset\mathbb R^k$, there exists $q\in U$ arbitrarily close to $p$ such that
\[\left|\frac{f(q)-f(p)}{|pq|}-\xi\right|<\varepsilon.\]
Then, $f$ is $(1-\varepsilon)$-open on $U$.
\end{lem}

\begin{proof}
Let $p\in U$ and $v\in\mathbb R^k$ be such that $\bar B(p,(1-\varepsilon)^{-1}|f(p)v|)\subset U$.
Consider a closed subset
\[A:=\bigl\{q\in X\mid (1-\varepsilon)|pq|\le|f(p)v|-|f(q)v|\bigr\}.\]
Note that $A\subset U$.
We show $v\in f(A)$.
Suppose the contrary and choose $q\in A$ such that $|f(q)v|=|f(A)v|>0$ by the properness of $X$.
Applying the assumption to $q$ and $\xi:=(v-f(q))/|f(q)v|$, we find $q_1\in U$ such that
\[\left|\frac{f(q_1)-f(q)}{|qq_1|}-\xi\right|<\varepsilon\quad\text{and}\quad0<|qq_1|\le|f(q)v|.\]
Then, we have
\begin{align*}
|f(q_1)v|&=|(v-f(q))-(f(q_1)-f(q))|\\
&=|(|f(q)v|\xi-|qq_1|\xi)+(|qq_1|\xi-(f(q_1)-f(q)))|\\
&\le(|f(q)v|-|qq_1|)+\varepsilon|qq_1|\\
&=|f(q)v|-(1-\varepsilon)|qq_1|.
\end{align*}
This implies
\[(1-\varepsilon)|pq_1|\le(1-\varepsilon)(|pq|+|pq_1|)\le|f(p)v|-|f(q_1)v|.\]
Therefore, $q_1\in A$ and $|f(q_1)v|<|f(q)v|$.
This contradicts the choice of $q$.
\end{proof}

The next lemma guarantees that an extremal subset almost contains the directions of strainers.
Recall that we denote by $S^k(\hat\Sigma)$ the $k$-fold spherical suspension over a space $\hat\Sigma$ of curvature $\ge 1$ and by $\{(\xi_i,\eta_i)\}_{i=1}^k$ an orthogonal $k$-frame of $S^k(\hat\Sigma)$.
We regard the standard unit sphere $\mathbb S^{k-1}$ as being embedded in $S^k(\hat\Sigma)$ (see \S\ref{sec:str}).

\begin{lem}\label{lem:susp}
Let $\Sigma$ be an $(n-1)$-dimensional Alexandrov space of curvature $\ge1$ and $F$ an extremal subset of $\Sigma$.
Suppose that $\Sigma$ has a $(k,\delta)$-strainer $\{(A_i,B_i)\}_{i=1}^k$.
Then, there exist an Alexandrov space $\hat\Sigma$ of curvature $\ge1$ and dimension $\le n-k-1$ (possibly empty) and a $\varkappa(\delta)$-Hausdorff approximation $\varphi:\Sigma\to S^k(\hat\Sigma)$ such that
\[\varphi(A_i)=\xi_i,\quad\varphi(B_i)=\eta_i\]
for all $1\le i\le k$ and the $\varkappa(\delta)$-neighborhood of $\varphi(F)$ contains $\mathbb S^{k-1}$.
\end{lem}

\begin{proof}
We prove it by contradiction.
Suppose that the claim does not hold for $\Sigma_j$, $F_j$, $\{(A_{ij},B_{ij})\}_{i=1}^k$ and $\delta_j\to0$, where $j\to\infty$.
Passing to a subsequence, we may assume that $\Sigma_j$ converges to an Alexandrov space $\Sigma$ of curvature $\ge1$ and dimension $\le n-1$ and that $F_j$ converges to an extremal subset $F$ of $\Sigma$.
We may further assume that $(k,\delta_j)$-strainers $\{(A_{ij},B_{ij})\}_{i=1}^k$ converge to an orthogonal $k$-frame $\{(\xi_i,\eta_i)\}_{i=1}^k$ of $\Sigma$.
Therefore, $\Sigma$ is isometric to $S^k(\hat\Sigma)$ for some $\hat\Sigma$ of curvature $\ge1$ and dimension $\le n-k-1$.
Consider the Euclidean cone $K(\Sigma)=\mathbb R^k\times K(\hat\Sigma)$.
Then, $K(F)$ is an extremal subset of $K(\Sigma)$.
Since every gradient curve starting at the vertex of $K(\Sigma)$ remains in $K(F)$, we see that $\mathbb R^k\times\{o\}\subset K(F)$, where $o$ denotes the vertex of $K(\hat\Sigma)$.
In other words, $\mathbb S^{k-1}\subset F$.
This is a contradiction.
\end{proof}

Now, we can prove the following:

\begin{prop}\label{prop:open}
Let $M$ be an Alexandrov space and $p\in M$.
Let $\{(a_i,b_i)\}_{i=1}^k$ be a $(k,\delta)$-strainer at $p$ with length $>\ell$.
Then, $f=(|a_1\cdot|,\dots,|a_k\cdot|):M\to \mathbb R^k$ is $(1+\varkappa(\delta))$-Lipschitz and $(1-\varkappa(\delta))$-open on $B(p,\ell\delta)$.

Moreover, let $E$ be an extremal subset of $M$ and assume $p\in E$.
Then, $f|_E$ is also $(1-\varkappa(\delta))$-open on $B(p,\ell\delta)\cap E$ with respect to the extrinsic metric of E.
\end{prop}

\begin{proof}
First, we show the $(1+\varkappa(\delta))$-Lipschitzness of $f$.
Note that $\{(a_i,b_i)\}_{i=1}^k$ is also a $(k,\varkappa(\delta))$-strainer for every point in $B(p,\ell\delta)$.
For any $x,y\in B(p,\ell\delta)$, we have
\[\left|\frac{|a_ix|-|a_iy|}{|xy|}-\cos\tilde\angle a_ixy\right|<\varkappa(\delta)\]
by the cosine formula on the $\kappa$-plane.
Thus,
\[\left|\frac{|f(x)f(y)|^2}{|xy|^2}-\sum_{i=1}^k\cos^2\tilde\angle a_ixy\right|<\varkappa(\delta).\]
Furthermore, by the property of a strainer, we have
\begin{equation}\label{eq:str}
\left|\tilde\angle a_ixy-\angle a_ixy\right|<\varkappa(\delta)
\end{equation}
(see \cite[5.6]{BGP}, \cite[10.8.13]{BBI}).
Let $(A_i)'_x$ and $(B_i)'_x$ be the sets of all directions of shortest paths from $x$ to $a_i$ and $b_i$, respectively.
Then, by Lemma \ref{lem:susp}, there exists a $\varkappa(\delta)$-Hausdorff approximation $\varphi:\Sigma_x\to S^k(\hat\Sigma)$ such that $\varphi((A_i)'_x)=\xi_i$ and $\varphi((B_i)'_x)=\eta_i$.
Hence, we have
\[\sum_{i=1}^k\cos^2\angle a_ixy<1+\varkappa(\delta)\]
since $\sum_{i=1}^k\cos^2\angle(\xi_i,\xi)\le1$ for any $\xi\in S^k(\hat\Sigma)$.
Combining the inequalities above, we obtain the $(1+\varkappa(\delta))$-Lipschitzness of $f$.

Next, we show the $(1-\varkappa(\delta))$-openness of $f|_E$.
It suffices to check that $f|_E$ satisfies the assumption of Lemma \ref{lem:open} on $B(p,\ell\delta)\cap E$.
Let $x\in B(p,\ell\delta)\cap E$ and $\varphi:\Sigma_x\to S^k(\hat\Sigma)$ be as above.
By Lemma \ref{lem:susp}, we may further assume that the $\varkappa(\delta)$-neighborhood of $\varphi(\Sigma_xE)$ contains $\mathbb S^{k-1}$ in $S^k(\hat\Sigma)$.
Recall that $f'(\eta)=\left(-\cos\angle((A_1)'_x,\eta),\dots,-\cos\angle((A_k)'_x,\eta)\right)$ for $\eta\in\Sigma_x$ by the first variation formula.
Therefore, for any $\xi\in\mathbb S^{k-1}\subset\mathbb R^k$, we can find $\eta\in\Sigma_xE$ such that 
\[\left|f'(\eta)-\xi\right|<\varkappa(\delta).\]
Thus, $f|_E$ satisfies the assumption of Lemma \ref{lem:open} at $x$.
\end{proof}

\begin{rem}\label{rem:open}
The first part of the above proposition is an improvement of \cite[5.4]{BGP}.
The second part is generalized to regular admissible maps defined in \cite{Per2}: the restriction of a regular admissible map to an extremal subset is $c$-open for some $c>0$ (see \cite[\S2]{PP1}, \cite[\S9]{K}).
Note that given a strainer, one can slightly modify it so that the associated distance map is regular as an admissible map.
Furthermore, it was pointed out by the referee that once $c$-openness is proved, one can obtain $(1-\varkappa(\delta))$-openness by using the infinitesimal characterization of such maps in \cite[1.2]{L} together with the limit argument as in Lemma \ref{lem:susp}.
See \cite[\S8.3]{LyN} for an analogous argument.
\end{rem}

As a corollary, we see that the Hausdorff dimension of an extremal subset is equal to the strainer number (see \S\ref{sec:str} for the definition).
Note that since the intrinsic and extrinsic metrics of an extremal subset are locally bi-Lipschitz equivalent, the Hausdorff dimensions in both metrics coincide.

\begin{cor}\label{cor:str}
The Hausdorff dimension of an extremal subset is equal to its $\delta$-strainer number for any small $\delta>0$.
\end{cor}

\begin{proof}
Let $p$ be a $(k,\delta)$-strained point of an extremal subset $E$.
Then, by Proposition \ref{prop:open}, there exists a Lipschitz map from a neighborhood of $p$ in $E$ onto an open subset in $\mathbb R^k$.
Thus, $\dim_{\mathrm H}E\ge\delta\mathchar`-\str(E)$.
On the other hand, Theorem \ref{thm:str} implies $\dim_{\mathrm H}E\le\delta\mathchar`-\str(E)$.
\end{proof}

From now, by the dimension of an extremal subset we mean its Hausdorff dimension (see also Corollary \ref{cor:dim} below).
Since a strainer is liftable under the Gromov-Hausdorff convergence of Alexandrov spaces, the above corollary implies

\begin{cor}\label{cor:lim}
The limit of $m$-dimensional extremal subsets under the Gromov-Hausdorff convergence of Alexandrov spaces is an extremal subset of dimension $\le m$.
\end{cor}

In particular, we obtain

\begin{cor}\label{cor:loc}
Let $E$ be an extremal subset of an Alexandrov space $M$ and $p\in E$.
Then, we have
\[\dim T_pE=\delta\mathchar`-\str_p(E),\quad\dim \Sigma_pE=\delta\mathchar`-\str_p(E)-1.\]
for any small $\delta>0$ (see \S\ref{sec:str} for the definition of the local strainer number).
\end{cor}

\begin{proof}
The inequality $\dim T_pE\le\delta\mathchar`-\str_p(E)$ follows in the same way as the previous corollary.
On the other hand, it is known that there exists a Lipschitz map (gradient exponential map) from $T_pE$ onto a neighborhood of $p$ in $E$ (\cite[6.2]{F1}).
Furthermore, by Proposition \ref{prop:open}, the Hausdorff dimension of any neighborhood of $p$ in $E$ is not less than the local $\delta$-strainer number of $E$ at $p$.
Thus, the opposite inequality follows.
\end{proof}

\begin{rem}\label{rem:pri}
If $E$ is primitive, then by Corollary \ref{cor:den} below, we have $\dim T_pE=\dim E$ and $\dim\Sigma_pE=\dim E-1$.
\end{rem}

Now, we define regular points of extremal subsets.
Since any extremal subset is uniquely represented as a union of primitive extremal subsets with nonempty relative interior, it is sufficient to define regular points of primitive extremal subsets (see \S\ref{sec:ext}).
The main reason why we use primitive extremal subsets is that the local strainer number of a general extremal subset at each point is not constant.

\begin{dfn}\label{dfn:reg}
Let $E$ be an $m$-dimensional primitive extremal subset of an Alexandrov space $M$.
A point $p\in E$ is said to be \textit{regular} if the tangent cone $T_pE$ equipped with the extrinsic metric is isometric to $\mathbb R^m$.
\end{dfn}

Then, the following holds:

\begin{prop}\label{prop:reg}
Let $E$ be an $m$-dimensional primitive extremal subset of an Alexandrov space $M$.
Then, $p\in E$ is regular if and only if it is $(m,\delta)$-strained for any $\delta>0$.
\end{prop}

\begin{proof}
The ``only if'' part is clear.
Conversely, if $p\in E$ is $(m,\delta)$-strained for any $\delta>0$, then we can find an orthogonal $m$-frame in $\Sigma_p$.
Therefore, $T_p$ splits isometrically into $\mathbb R^m\times K(\hat\Sigma)$ for some $\hat\Sigma$ of curvature $\ge1$.
Then, we have $\mathbb R^m\times\{o\}\subset T_pE$ by extremality as in the proof of Lemma \ref{lem:susp}.
Furthermore, if there exists $(v,r\xi)\in T_pE\setminus\mathbb R^m\times\{o\}$, where $v\in\mathbb R^m$, $\xi\in\hat\Sigma$ and $r>0$, a similar argument using gradient curves shows that $\mathbb R^m\times K(\{\xi\})\subset T_pE$.
However, this contradicts $\dim T_pE=m$.
Thus, we conclude that $T_pE=\mathbb R^m\times\{o\}$.
\end{proof}

Recall that for an Alexandrov space $M$, we denote by $M(k,\delta)$ the set of $(k,\delta)$-strained points in $M$ and by $M(k)$ the intersection of $M(k,\delta)$ for all $\delta>0$ (see \S\ref{sec:str}).
For an extremal subset $E$ of $M$, set $E(k,\delta):=E\cap M(k,\delta)$ and $E(k):=E\cap M(k)$.
Then, the above proposition together with Theorem \ref{thm:str} implies the following:

\begin{cor}\label{cor:reg}
Let $E$ be an $m$-dimensional primitive extremal subset of an Alexandrov space $M$.
Then, $E(m)$ is equal to the set of regular points in $E$.
In particular, the Hausdorff dimension of the set of non-regular points in $E$ is no more than $m-1$.
\end{cor}

Next, we prove the main theorem of this paper:

\begin{thm}\label{thm:reg}
Let $M$ be an Alexandrov space, $E\subset M$ an extremal subset and $p\in E$.
Let $\{(a_i,b_i)\}_{i=1}^k$ be a $(k,\delta)$-strainer at $p$ such that $k$ is the local $\delta$-strainer number of $E$ at $p$.
Then, $f=(|a_1\cdot|,\dots,|a_k\cdot|)$ gives a $\varkappa(\delta)$-almost isometry from a neighborhood of $p$ in $E$ to an open subset in $\mathbb R^k$ with respect to both intrinsic and extrinsic metrics of $E$.
\end{thm}

\begin{rem}\label{rem:reg1}
If $E$ is primitive, then by Corollary \ref{cor:den} below, the local strainer number of $E$ at any point is equal to $\dim E$.
\end{rem}

For the proof, we need the maximal case of Lemma \ref{lem:susp}:

\begin{lem}\label{lem:suspm}
Under the assumption of Lemma \ref{lem:susp}, suppose that $\dim F=k-1$.
Then, we have $d_{\mathrm H}(\varphi(F),\mathbb S^{k-1})<\varkappa(\delta)$ and $\diam\hat\Sigma\le\pi/2$ in the conclusion, where $d_\mathrm{H}$ denotes the Hausdorff distance in $S^k(\hat\Sigma)$.
\end{lem}

\begin{proof}
In view of Corollary \ref{cor:lim} and the proof of Lemma \ref{lem:susp}, it suffices to show that if $K(F)$ is an extremal subset of dimension $\le k$ in $K(\Sigma)=\mathbb R^k\times K(\hat\Sigma)$, then $K(F)=\mathbb R^k\times\{o\}$ and $\diam\hat\Sigma\le\pi/2$.
The former follows from the same argument as in the proof of Proposition \ref{prop:reg}.
The latter follows from the extremality of $K(F)$.
\end{proof}

\begin{proof}[Proof of Theorem \ref{thm:reg}]
Let $\ell>0$ be a lower bound of the length of the strainer $\{(a_i,b_i)\}_{i=1}^k$.
Let $U$ be a sufficiently small open neighborhood of $p$ in $E$ such that
\begin{enumerate}
\item the diameter of $U$ in the intrinsic metric of $E$ is less than $\ell\delta$;
\item there are no $(k+1,\delta)$-strained points in $U$.
\end{enumerate}
Note that $\{(a_i,b_i)\}_{i=1}^k$ is a $(k,\varkappa(\delta))$-strainer for every point in $U$.
Furthermore, by Proposition \ref{prop:open}, $f$ is $(1+\varkappa(\delta))$-Lipschitz on $U$ and $f(U)$ is an open subset of $\mathbb R^k$.
Let $x,y\in U$ and let $|\ ,\ |_E$ denote the intrinsic metric of $E$.
Since
\[\frac{|f(x)f(y)|}{|xy|_E}\le\frac{|f(x)f(y)|}{|xy|}<1+\varkappa(\delta),\]
it remains to verify that
\begin{equation}\label{eq:bilip}
\frac{|f(x)f(y)|}{|xy|_E}>1-\varkappa(\delta).
\end{equation}

According to the generalized Lieberman theorem, a shortest path between $x$ and $y$ in the intrinsic metric of $E$ is a quasigeodesic of $M$ (see \S\ref{sec:qg}).
First, by the condition (1) and the cosine formula on the $\kappa$-plane, we have
\[\left|\frac{|a_ix|-|a_iy|}{|xy|_E}-\cos\tilde\angle a_ix\!\smile\!y\right|<\varkappa(\delta)\]
where $\tilde\angle a_ix\!\smile\!y$ denotes the comparison angle of a quasigeodesic defined in \S\ref{sec:qg}.
Therefore,
\begin{equation}\label{eq:A}
\left|\frac{|f(x)f(y)|^2}{|xy|_E^2}-\sum_{i=1}^k\cos^2\tilde\angle a_ix\!\smile\!y\right|<\varkappa(\delta),
\end{equation}
Next, we show that
\begin{equation}\label{eq:B}
\left|\tilde\angle a_ix\!\smile\!y-\angle((a_i)'_x,y^\circ_x)\right|<\varkappa(\delta),
\end{equation}
where $y^\circ_x\in\Sigma_xE$ denotes the direction of a shortest path in the intrinsic metric of $E$.
The proof is the same as that of the inequality \eqref{eq:str} except that the monotonicity of comparison angles of quasigeodesics is used.
Firstly, the condition (1) and the Gauss-Bonnet formula on the $\kappa$-plane imply that
\begin{equation}\label{eq:a}
\begin{gathered}
\left|\tilde\angle a_ix\!\smile\!y+\tilde\angle a_iy\!\smile\!x-\pi\right|<\varkappa(\delta),\\
\left|\tilde\angle b_ix\!\smile\!y+\tilde\angle b_iy\!\smile\!x-\pi\right|<\varkappa(\delta).
\end{gathered}
\end{equation}
Secondly, the monotonicity of comparison angles of quasigeodesics yields
\begin{equation}\label{eq:b}
\begin{gathered}
\tilde\angle a_ix\!\smile\!y\le\angle((a_i)'_x,y^\circ_x),\quad\tilde\angle a_iy\!\smile\!x\le\angle((a_i)'_y,x^\circ_y),\\
\tilde\angle b_ix\!\smile\!y\le\angle((b_i)'_x,y^\circ_x),\quad\tilde\angle b_iy\!\smile\!x\le\angle((b_i)'_y,x^\circ_y).
\end{gathered}
\end{equation}
Finally, since $(a_i,b_i)$ is a $(1,\varkappa(\delta))$-strainer for $x$ and $y$, we have
\begin{equation}\label{eq:c}
\begin{gathered}
\angle((a_i)'_x,y^\circ_x)+\angle((b_i)'_x,y^\circ_x)\le2\pi-\angle((a_i)'_x,(b_i)'_x)<\pi+\varkappa(\delta),\\
\angle((a_i)'_y,x^\circ_y)+\angle((b_i)'_y,x^\circ_y)\le2\pi-\angle((a_i)'_y,(b_i)'_y)<\pi+\varkappa(\delta).
\end{gathered}
\end{equation}
Combining the inequalities \eqref{eq:a}, \eqref{eq:b} and \eqref{eq:c}, we obtain
\begin{align*}
&2\pi-2\varkappa(\delta)\\
&<\tilde\angle a_ix\!\smile\!y+\tilde\angle a_iy\!\smile\!x+\tilde\angle b_ix\!\smile\!y+\tilde\angle b_iy\!\smile\!x\\
&\le\angle((a_i)'_x,y^\circ_x)+\angle((a_i)'_y,x^\circ_y)+\angle((b_i)'_x,y^\circ_x)+\angle((b_i)'_y,x^\circ_y)\\
&<2\pi+2\varkappa(\delta).
\end{align*}
Therefore, the difference between the both sides of each inequality in \eqref{eq:b} is at most $4\varkappa(\delta)$.
In particular, we obtain the inequality \eqref{eq:B}.

Recall that $\{((a_i)'_x,(b_i)'_x)\}_{i=1}^k$ is a $(k,\varkappa(\delta))$-strainer of $\Sigma_x$.
In addition, the condition (2) together with Corollary \ref{cor:loc} implies that $\dim\Sigma_xE= k-1$.
Hence, it follows from Lemma \ref{lem:suspm} that there exists a $\varkappa(\delta)$-Hausdorff approximation $\varphi:\Sigma_x\to S^k(\hat\Sigma)$ such that
\[\varphi((a_i)'_x)=\xi_i,\quad\varphi((b_i)'_x)=\eta_i\quad\text{and}\quad d_{\mathrm H}(\varphi(\Sigma_xE),\mathbb S^{k-1})<\varkappa(\delta),\]
where $\{(\xi_i,\eta_i)\}_{i=1}^k$ is an orthogonal $k$-frame of $S^k(\hat\Sigma)$.
Therefore, we have
\begin{equation}\label{eq:C}
\left|\sum_{i=1}^k\cos^2\angle((a_i)'_x,y^\circ_x)-1\right|<\varkappa(\delta)
\end{equation}
since $\sum_{i=1}^k\cos^2\angle(\xi_i,\xi)=1$ for any $\xi\in\mathbb S^{k-1}$.
Combining the inequalities \eqref{eq:A}, \eqref{eq:B} and \eqref{eq:C}, we obtain the desired inequality \eqref{eq:bilip}.
\end{proof}

\begin{rem}\label{rem:reg2}
As with Proposition \ref{prop:open}, it was pointed out by the referee that Theorem \ref{thm:reg} also follows from the infinitesimal characterization of bi-Lipschitz maps in \cite[1.3]{L}.
Roughly speaking, the result of \cite[1.3]{L} states that to prove that $f|_E$ is $\varkappa(\delta)$-almost isometry, it suffices to show that so is its derivative at each point.
The latter follows from an argument by contradiction similar to the proof of Lemma \ref{lem:suspm}.
See \cite[\S8.3]{LyN} for an analogous argument.
Note that the generalized Lieberman theorem is still needed to obtain the $(1-\varkappa(\delta))$-openness with respect to the intrinsic metric of $E$.
Indeed, to show that the blow up of $E$ at $x$ with respect to the intrinsic metric coincides with $T_xE$ with the intrinsic metric, one needs the convergence theorem \cite[1.2]{Pet1} for the intrinsic metrics of extremal subsets, which uses the generalized Lieberman theorem.
\end{rem}

The rest of this section consists of corollaries of Theorem \ref{thm:reg}.
First, the intrinsic and extrinsic metrics are almost equal near a regular (strained) point in the following sense:

\begin{cor}\label{cor:met}
Under the assumption of Theorem \ref{thm:reg}, we have
\[\frac{|xy|_E}{|xy|}<1+\varkappa(\delta),\quad\angle(y'_x,y^\circ_x)<\varkappa(\delta)\]
for any $x,y\in E$ sufficiently close to $p$.
\end{cor}

\begin{proof}
The first inequality is clear from Theorem \ref{thm:reg}.
In particular, we have $|\tilde\angle a_ixy-\tilde\angle a_ix\!\smile\!y|<\varkappa(\delta)$.
Together with the inequalities \eqref{eq:str} and \eqref{eq:B}, this implies $|\angle((a_i)'_x,y'_x)-\angle((a_i)'_x,y^\circ_x)|<\varkappa(\delta)$.
Therefore, the Hausdorff approximation $\varphi:\Sigma_x\to S^k(\hat\Sigma)$ sends $y'_x$ and $y^\circ_x$ to two $\varkappa(\delta)$-close points near $\mathbb S^{k-1}$.
\end{proof}

Together with Corollary \ref{cor:reg}, the above implies

\begin{cor}\label{cor:vol}
The $m$-dimensional Hausdorff measures on an $m$-dimensional extremal subset with respect to the intrinsic and extrinsic metrics coincide.
\end{cor}

\begin{proof}
Let $E$ be an $m$-dimensional extremal subset.
Let $d$ and $d_E$ denote the extrinsic and intrinsic metrics of $E$, respectively.
We denote by $\vol_m$ the $m$-dimensional Hausdorff measure.
Since $d$ and $d_E$ are bi-Lipschitz equivalent and thus absolutely continuous with respect to each other, by the Radon-Nikodym theorem, it suffices to show that
\[\lim_{r\to0}\frac{\vol_m(B(p,r),d_E)}{\vol_m(B(p,r),d)}=1\]
for almost all $p\in E$, where the ball $B(p,r)$ is with respect to the extrinsic metric.
By Corollary \ref{cor:reg}, we may assume that $p\in E(m)$.
Then, the above equality follows from Corollary \ref{cor:met}.
\end{proof}

Theorem \ref{thm:reg} implies the denseness of regular (strained) points:

\begin{cor}\label{cor:den}
The set of $(m,\delta)$-strained points in an $m$-dimensional primitive extremal subset is dense.
Furthermore, the set of regular points is also dense.
\end{cor}

\begin{proof}
Let $E$ be an $m$-dimensional primitive extremal subset and $\mathring E$ its main part (see \S\ref{sec:ext}).
Recall that $\mathring E$ is open and dense in $E$, and is a topological manifold.
Fix small $\delta>0$ and set $k_p:=\delta\mathchar`-\str_p(E)$ for $p\in E$.
Then, Theorem \ref{thm:reg} implies that there exists an open subset in $E$ homeomorphic to $\mathbb R^{k_p}$.
Therefore, by the properties of $\mathring E$ above, $k_p$ must be a constant independent of $p$ and thus it equals $m=\delta\mathchar`-\str(E)$.
Hence, $E(m,\delta)$ is dense in $E$ for any $\delta>0$ and $E(m)$ is also dense as the intersection of countably many open dense subsets.
\end{proof}

The above proof also shows

\begin{cor}\label{cor:dim}
The dimension of a primitive extremal subset is equal to the dimension of its main part as a topological manifold.
Moreover, it coincides with the maximal dimension of open subsets homeomorphic to Euclidean balls.
\end{cor}

Lastly, we describe a neighborhood of a regular (strained) point of an extremal subset in the ambient Alexandrov space.
Let $\rho>0$.
For a subset $A$ in a metric space $X$, we denote by $U_\rho(A)$ (resp.\ $\bar U_\rho(A)$) the open (resp.\ closed) $\rho$-neighborhood of $A$, and define $\partial U_\rho(A):=\bar U_\rho(A)\setminus U_\rho(A)$.
For $v\in\mathbb R^k$, we denote by $I^k_\rho(v)$ (resp.\ $\mathring I^k_\rho(v)$) the closed (resp.\ open) $\rho$-neighborhood of $v$ with respect to the maximum norm of $\mathbb R^k$.
Furthermore, for a topological space $\Sigma$, we denote by $K_\rho(\Sigma)$ the quotient of $\Sigma\times[0,\rho)$ by identifying all points in $\Sigma\times\{0\}$.
Note that there is a natural projection from $K_\rho(\Sigma)$ to $[0,\rho)$.

\begin{cor}\label{cor:lnhd}
Under the assumption of Theorem \ref{thm:reg}, suppose that $\delta$ is sufficiently small (depending on $p$).
Then, for sufficiently small $r\gg\rho>0$, there exists a homeomorphism
\[\Phi:f^{-1}(\mathring I^k_\rho(f(p))\cap U_\rho(E)\cap B(p,r)\to\mathring I^k_\rho(f(p))\times K_\rho(\Sigma)\]
respecting $(f,|E\cdot|)$, that is, $(f,|E\cdot|)\equiv\mathrm{pr}\circ\Phi$, where $\mathrm{pr}:\mathring I^k_\rho(f(p))\times K_\rho(\Sigma)\to\mathring I^k_\rho(f(p))\times[0,\rho)$ is the natural projection and $\Sigma=f^{-1}(f(p))\cap\partial U_\rho(E)\cap B(p,r)$.
\end{cor}

We need the following lemma:

\begin{lem}[{\cite[3.1(2)]{PP1}}, {\cite[4.1.4]{Pet2}}]\label{lem:grad}
Let $M$ be an $n$-dimensional Alexandrov space with curvature $\ge\kappa$, $E\subset M$ an extremal subset and $p\in E$ such that $\vol_nB(p,D)\ge v>0$.
Then, there exists $\varepsilon>0$ depending only on $n$, $\kappa$, $D$ and $v$ such that $|\nabla\dist_E|>\varepsilon$ on $(U_\varepsilon (E)\setminus E)\cap B(p,D)$.
\end{lem}

\begin{proof}[Proof of Corollary \ref{cor:lnhd}]
Take $\varepsilon>0$ such that $|\nabla_x\dist_E|>\varepsilon$ and  $\vol_{n-1}\Sigma_x>\varepsilon$ for any $x\in B(p,\varepsilon)\setminus E$ (note that the choice of $\varepsilon$ depends on the volume of a small neighborhood of $p$).
Let $\delta\ll\varepsilon$ and let $0<r<\min\{\varepsilon,\ell\delta\}$ be so small that $B(p,r)$ is contained in $U$ in the proof of Theorem \ref{thm:reg}.

First, we observe that $(f,|E\cdot|)$ is noncritical on $B(p,r)\setminus E$ in the sense of \cite[3.1, 3.7]{Per1}.
Let $(A_i)'_x$ denote the set of all directions from $x$ to $a_i$.
We show that there exist positive numbers $\mu\ll\nu$ such that for any $x\in B(p,r)\setminus E$, we have $\vol_{n-1}\Sigma_x>\nu$ and
\[\angle((A_i)'_x,(A_j)'_x) ,\angle((A_i)'_x,E'_x)>\pi/2-\mu\]
for all $i\neq j$, and there exists a direction $\xi\in\Sigma_x$ such that
\[\angle((A_i)'_x,\xi),\angle(E'_x,\xi)>\pi/2+\nu\]
for all $i$ (although the original definition of noncritical maps in \cite[3.1]{Per1} uses comparison angles, it can be weakened to angles).

Indeed, $\angle((A_i)'_x,(A_j)'_x)>\pi/2-\varkappa(\delta)$ since $r<\ell\delta$.
Let $y\in E$ be a closest point to $x$.
Then, we have $\tilde\angle a_iyx\le\pi/2$ by the extremality of $E$.
Furthermore, $|\tilde\angle a_ixy+\tilde\angle a_iyx-\pi|<\varkappa(\delta)$ since $r<\ell\delta$.
Therefore, $\tilde\angle a_ixy>\pi/2-\varkappa(\delta)$.
This implies $\angle((A_i)'_x,E'_x)>\pi/2-\varkappa(\delta)$.
Similarly, we have $\angle((B_i)'_x,E'_x)>\pi/2-\varkappa(\delta)$, where $(B_i)'_x$ denotes the set of all directions from $x$ to $b_i$.
In other words, $E'_x$ is almost orthogonal to $(A_i)'_x$ and $(B_i)'_x$.
Since $|\nabla_x\dist_E|>\varepsilon$, there exists $\eta\in\Sigma_x$ such that $\angle(E'_x,\eta)>\pi/2+\varepsilon$.
Then, by using Lemma \ref{lem:susp} (or moving $\eta$ towards $(B_i)'_x$'s), we can find $\zeta\in\Sigma_x$ such that $\angle((A_i)'_x,\zeta),\angle(E'_x,\zeta)>\pi/2+c(\varepsilon)$, where $c(\varepsilon)$ is a constant depending only on $n$ and $\varepsilon$ (note $\varkappa(\delta)\ll c(\varepsilon)$).
Hence, $(f,|E\cdot|)$ is noncritical on $B(p,r)\setminus E$.

Next, we observe that $(f,|E\cdot|)$ is proper on $f^{-1}(\mathring I^k_\rho(f(p)))\cap(U_\rho(E)\setminus E)\cap B(p,r)$ for sufficiently small $\rho\ll r$.
It is enough to show that $f^{-1}(I^k_\rho(f(p)))\cap\bar U_\rho(E)\cap B(p,r)$ is compact.
Let $x\in f^{-1}(I^k_\rho(f(p)))\cap\bar U_\rho(E)\cap B(p,r)$.
Let $y\in E$ be a closest point to $x$ and let $z\in E$ be the point such that $f(z)=f(x)$ (it exists by Theorem \ref{thm:reg}).
Then, we have $|xy|\le\rho$ and $|pz|<(1+\varkappa(\delta))|f(p)f(x)|<2\sqrt k\rho$.
Furthermore, we have $|yz|<(1+\varkappa(\delta))|f(y)f(x)|<(1+\varkappa(\delta))^2|yx|<2\rho$.
Therefore, we obtain
\[|xp|\le|xy|+|yz|+|zp|\le(3+2\sqrt k)\rho.\]
Thus, $f^{-1}(I^k_\rho(f(p)))\cap\bar U_\rho(E)\cap B(p,r)$ is compact provided that $(3+2\sqrt k)\rho<r$.

Therefore, the fibration theorem \cite[1.4.1]{Per1} implies that there exists a homeomorphism $\Psi:f^{-1}(\mathring I^k_\rho(f(p)))\cap(U_\rho(E)\setminus E)\cap B(p,r)\to\mathring I^k_\rho(f(p))\times(0,\rho)\times\Sigma$ respecting $(f,|E\cdot|)$.
Extending this homeomorphism to $f^{-1}(\mathring I^k_\rho(f(p)))\cap E\cap B(p,r)$ by the almost isometry $f$, we obtain the desired homeomorphism.
\end{proof}

\section{Applications}\label{sec:appl}

In this section, we prove Theorem \ref{thm:vol} and Theorem \ref{thm:nhd}.

First, we prove Theorem \ref{thm:vol}.
Note that by Corollary \ref{cor:vol}, the Hausdorff measure on an extremal subset is independent of whether the metric is intrinsic or extrinsic.
The following proof is analogous to that of the measured convergence for spaces with an upper curvature bound in \cite[\S12]{LyN}.
The author is grateful to the referee for suggesting this proof.

\begin{proof}[Proof of Theorem \ref{thm:vol}]
We focus on where the noncollapsing assumption is used.
Let $M$ be an $n$-dimensional Alexandrov space with curvature $\ge\kappa$ and $E$ an extremal subset.
In the following argument, we use the fact that if $\vol_n(B(p,D))\ge v>0$ for $p\in E$, then there exists a constant $C=C(n,\kappa,D,v)$ such that the intrinsic and extrinsic metrics of $E$ are $C$-Lipschitz equivalent on $B(p,D)\cap E$ (see \cite[\S4.1 property 4]{Pet2} or \cite[3.2(2)]{PP1}).
We also use the fact that $\vol_m(B(p,D)\cap E)$ is uniformly bounded above by $C(n,\kappa,D)$ (see \cite[1.4]{LiN} or \cite[1.1(3)]{F1}).

Let $(M_i,E_i)\overset{\mathrm{GH}}\longrightarrow (M,E)$ be a noncollapsing sequence of Alexandrov spaces and $m$-dimensional extremal subsets as in Theorem \ref{thm:vol}.
For simplicity, we assume $E$ is compact (the noncompact case is similar).
Since $\vol_m E_i$ is uniformly bounded above, by passing to a subsequence, we may assume that $(E_i,\vol_m)$ converges to $(E,\mu)$ in the measured Gromov-Hausdorff topology, where $\mu$ is a Radon measure on $E$.
Let us prove $\mu=\vol_m$.
By rescaling, $\vol_m(B(p_i,r)\cap E_i)\le Cr^m$ for any $p_i\in E_i$ and $0<r\le 1$.
This implies that $\mu$ is absolutely continuous with respect to $\vol_m$.
Thus, by the Radon-Nikodym theorem, it suffices to show that
\[\lim_{r\to0}\frac{\mu(B(p,r))}{\vol_m(B(p,r))}=1\]
for almost all $p\in E$.
By Corollary \ref{cor:reg}, we may assume $p\in E(m)$.
Then, there exists a $(m,\delta)$-strainer at $p$ for any small $\delta>0$.
Take a $(m,\delta)$-strainer at $p$ with length $>\ell$.
Recall that by the noncollapsing assumption, the intrinsic and extrinsic metrics of $E_i$ and $E$ are uniformly $C$-Lipschitz equivalent.
Let $0<r<C^{-1}\ell\delta$.
Then, $B(p,r)\cap E$ satisfies the condition (1) in the proof of Theorem \ref{thm:reg}.
Thus, by Theorem \ref{thm:reg}, we have
\[\left|\frac{\vol_m(B(p,r))}{v_m(r)}-1\right|<\varkappa(\delta),\]
where $v_m(r)$ is the volume of the $m$-dimensional Euclidean ball of radius $r>0$.
Let $p_i\in E_i$ be a sequence converging to $p$ and lift the $(m,\delta)$-strainer to $M_i$.
Then, by the choice of $r$ above, $B(p_i,r)\cap E_i$ still satisfies the condition (1) in the proof of Theorem \ref{thm:reg}.
Thus, by Theorem \ref{thm:reg} again, we have
\[\left|\frac{\vol_m(B(p_i,r))}{v_m(r)}-1\right|<\varkappa(\delta).\]
Passing to the limit and taking $\delta\to0$ (so $r\to0$) yield the desired equality.
\end{proof}

\begin{rem}\label{rem:vol}
Theorem \ref{thm:vol} also can be proved in the same way as the volume convergence of Alexandrov spaces (\cite[10.8]{BGP}, \cite[3.5]{S}, \cite[0.6]{Y1}).
Here is an outline.
Let $E(m,\delta,\ell)$ denote the set of points in $E$ having a $(m,\delta)$-strainer with length $>\ell$.
Let $p_i\in E_i$ be a sequence converging to $p\in E(m,\delta,\ell)$.
As in the above proof, by the noncollapsing assumption and Theorem \ref{thm:reg}, there exists a local almost isometry between neighborhoods of $p$ and $p_i$.
Gluing such local almost isometries, one can construct a global almost isometry from $E(m,\delta,\ell)$ into $E_i$ whose image contains $E_i(m,\delta/2,2\ell)$ (cf.\ \cite[9.8]{BGP} (see also \cite{WSS}), \cite[3.1]{S}, \cite[0.4]{Y1}).
Furthermore, the volumes of $E\setminus E(m,\delta,\ell)$ and $E_i\setminus E_i(m,\delta/2,2\ell)$ uniformly go to $0$ as $\ell\to 0$ (cf.\ \cite[1.3]{LiN}, \cite[6.6]{F1}, \cite[10.9]{BGP}).
Combining these two results yields the volume convergence of extremal subsets.
\end{rem}

Next, we prove Theorem \ref{thm:nhd}, which is a global version of Corollary \ref{cor:lnhd}.
Here is a more detailed formulation.
In view of Corollary \ref{cor:met}, we use the extrinsic metric below.

\begin{thm}[cf.\ {\cite[0.2]{Y1}} (see also \cite{F2})]\label{thm:gnhd}
Let $E$ be an $m$-dimensional compact extremal subset of an Alexandrov space $M$.
Suppose that every point in $E$ is $(m,\delta)$-strained, where $\delta$ depends on $M$.
Then, for sufficiently small $\rho>0$, there exists a locally trivial fibration $f:U_\rho(E)\to E$ such that
\begin{enumerate}
\item $f$ is $(1+\varkappa(\delta))$-Lipschitz and $(1-\varkappa(\delta))$-open;
\item $f$ is the identity on $E$;
\item for any $p\in E$, there exist a neighborhood $V\subset E$ and a homeomorphism $f^{-1}(V)\to V\times K_\rho(\Sigma)$ respecting $(f,|E\cdot|)$, where $\Sigma=f^{-1}(p)\cap \partial U_\rho(E)$ (see Corollary \ref{cor:lnhd} for the notation and terminology).
\end{enumerate}
\end{thm}

\begin{rem}\label{rem:subm}
Indeed, $f$ is locally a $\varkappa(\delta)$-almost Lipschitz submersion in the sense of \cite{Y1}.
See Claim \ref{clm:2} below and \cite[6.6]{F2}.
\end{rem}

The construction of $f$ in the following proof is analogous to \cite[3.1]{S} (cf.\ \cite[4.1]{F2}).

\begin{proof}[Proof of Theorem \ref{thm:gnhd}]
By the compactness of $E$, there exists $\ell>0$ such that every point in $E$ has an $(m,\delta)$-strainer with length $>\ell$.
Let $0<r<\ell\delta^2$ and take a maximal $r/2$-discrete set $\{p_j\}_{j=1}^N$ in $E$.
Let $\{(\alpha_i^j,\beta_i^j)\}_{i=1}^m$ be an $(m,\delta)$-strainer at $p_j$ with length $>\ell$.
Choose points $a_i^j$, $b_i^j$ on shortest paths $p_j\alpha_i^j$, $p_j\beta_i^j$ at distance $\ell\delta$ from $p_j$, respectively.
Then, $\{(a_i^j,b_i^j)\}_{i=1}^m$ is also an $(m,\delta)$-strainer at $p_j$.
Set
\begin{align*}
&\varphi_j:=(|a_1^j\cdot|,\dots,|a_m^j\cdot|),&&U_j:=B(p_j,r),&&\lambda U_j:=B(p_j,\lambda r)\\
&\psi_j:=\varphi_j|_E,&&V_j:=U_j\cap E,&&\lambda V_j:=\lambda U_j\cap E
\end{align*}
for $\lambda>0$.
Then, $\varphi_j$ is $(1+\varkappa(\delta))$-Lipschitz on $3U_j$ and $\psi_j$ is a $\varkappa(\delta)$-almost isometry from $10V_j$ to an open subset in $\mathbb R^m$.
Therefore, we can define $\psi_j^{-1}\circ\varphi_j$ on $3U_j$.
We take an average of them to obtain a global map.
Define $f_k:\bigcup_{l=1}^kU_l\to E$ for $1\le k\le N$ inductively as follows:
\[f_1:=\psi_1^{-1}\circ\varphi_1:U_1\to E,\]
and for $k\ge2$,
\begin{equation*}
f_k:=
\begin{cases}
\hfil f_{k-1} & \text{on}\hfil\bigcup_{l=1}^{k-1}U_l\setminus 2U_k,\\
\psi_k^{-1}((1-\chi_k)\psi_k\circ f_{k-1}+\chi_k\varphi_k) & \text{on}\quad\bigcup_{l=1}^{k-1}U_l\cap(2U_k\setminus U_k),\\
\hfil \psi_k^{-1}\circ\varphi_k & \text{on}\hfil U_k.
\end{cases}
\end{equation*}
Here, $\chi_k:=\chi(|p_k\cdot|/r)$, where $\chi:[0,\infty)\to[0,1]$ is a smooth function such that $\chi\equiv1$ on $[0,1]$ and $\chi\equiv0$ on $[2,\infty)$.
Note that $\chi_k$ is $C/r$-Lipschitz for some constant $C$.

Now, we show the following two claims for $f_k$ by induction on $k$.
Note that since the multiplicity of the covering $\{2U_j\}_{j=1}^N$ is bounded above by some constant depending only on $m$, so is the number of induction steps at each point of the domain of $f_k$.

First, we prove
\begin{clm}\label{clm:1}
For every $1\le j\le N$ and $1\le k\le N$, we have
\[\bigl|\psi_j^{-1}\circ\varphi_j,f_k\bigr|<\varkappa(\delta)r\]
on $3U_j\cap\bigcup_{l=1}^kU_l$ (and therefore, $f_{k+1}$ can be defined as above).
\end{clm}

\begin{proof}
For $x\in M$, we denote by $g(x)$ a point in $E$ nearest to $x$.
We show that
\begin{enumerate}
\item $|\psi_j^{-1}\circ\varphi_j(x),g(x)|<\varkappa(\delta)|xg(x)|$ for any $x\in3U_j$;
\item $|f_k(x),g(x)|<\varkappa(\delta)|xg(x)|$ for any $x\in\bigcup_{l=1}^kU_l$.
\end{enumerate}
We first prove (1).
Let $x\in 3U_j$ and fix $1\le i\le m$.
Then, by the extremality of $E$, we have $\tilde\angle a_i^jg(x)x,\tilde\angle b_i^jg(x)x\le\pi/2$.
Since the sum of these two comparison angles is almost $\pi$, we have $|\tilde\angle a_i^jg(x)x-\pi/2|<\varkappa(\delta)$.
Together with $|xg(x)|<3r\ll|a_i^jx|$, this implies $||a_i^jx|-|a_i^jg(x)||<\varkappa(\delta)|xg(x)|$.
Therefore, we have $|\varphi_j(x)\psi_j(g(x))|<\varkappa(\delta)|xg(x)|$.
Since $\psi_j$ is a $\varkappa(\delta)$-almost isometry, (1) holds.
(2) easily follows from (1) by induction on $k$.
\end{proof}

Note that the above (2) implies $|xf_k(x)|<2r\ll\ell\delta$ for any $x\in\bigcup_{l=1}^kU_l$.

Next, we prove
\begin{clm}\label{clm:2}
For every $1\le j\le N$ and $1\le k\le N$, we have
\begin{equation}\label{eq:clm2}
\bigl|(\varphi_j(x)-\varphi_j(y))-(\psi_j\circ f_k(x)-\psi_j\circ f_k(y))\bigr|<\varkappa(\delta)|xy|
\end{equation}
for any $x,y\in3U_j\cap\bigcup_{l=1}^kU_l$.
\end{clm}

\begin{proof}
Note that we may assume $|xy|<r$ by Claim \ref{clm:1}.
We use induction on $k$.
First, we consider the case $j=k$.
The base case $k=1$ is trivial.
Suppose $k\ge2$.
Let $x,y\in \bigcup_{l=1}^{k-1}U_l\cap(2U_k\setminus U_k)$ (the other case is similar).
Then, by the definition of $f_k$,
\begin{align*}
&(\varphi_k(x)-\varphi_k(y))-(\psi_k\circ f_k(x)-\psi_k\circ f_k(y))\\
&=(1-\chi_k(x))(\varphi_k(x)-\psi_k\circ f_{k-1}(x))-(1-\chi_k(y))(\varphi_k(y)-\psi_k\circ f_{k-1}(y))\\
&=(1-\chi_k(x))((\varphi_k(x)-\varphi_k(y))-(\psi_k\circ f_{k-1}(x)-\psi_k\circ f_{k-1}(y)))\\
&\qquad\qquad\qquad\qquad\qquad\qquad\qquad-(\chi_k(x)-\chi_k(y))(\varphi_k(y)-\psi_k\circ f_{k-1}(y)).
\end{align*}
The norm of the first term of the last formula is less than $\varkappa(\delta)|xy|$ by the induction hypothesis.
The same is true for the second term by Claim \ref{clm:1} and the $C/r$-Lipschitzness of $\chi_k$.

Next, we consider the case $j\neq k$.
Note that the inequality \eqref{eq:clm2} is equivalent to
\begin{equation}\label{eq:clm2'}
\left|\cos\angle a_i^jxy-\cos\angle a_i^jf_k(x)f_k(y)\cdot\frac{|f_k(x)f_k(y)|}{|xy|}\right|<\varkappa(\delta)
\end{equation}
for all $1\le i\le m$ (use the property \eqref{eq:str} twice).
By the induction hypothesis, we may assume $x,y\in3U_j\cap\bigcup_{l=1}^kU_l\cap 3U_k$ since $|xy|<r$.
Fix $1\le i\le m$.
Firstly, by Lemma \ref{lem:suspm}, there exists a $\varkappa(\delta)$-Hausdorff approximation $\Phi_1:\Sigma_{f_k(x)}\to S^m(\hat\Sigma_1)$ which sends $\{(a_h^k)'_{f_k(x)},(b_h^k)'_{f_k(x)})\}_{h=1}^m$ to an orthogonal $m$-frame.
Furthermore, $\Phi_1$ sends $(a_i^j)'_{f_k(x)}$ into the $\varkappa(\delta)$-neighborhood of $\mathbb S^{m-1}$ in $S^m(\hat\Sigma_1)$ since $((a_i^j)'_{f_k(x)},(b_i^j)'_{f_k(x)})$ is a $(1,\varkappa(\delta))$-strainer and $\diam\hat\Sigma_1\le\pi/2$.
Therefore, we have
\[\left|\cos\angle a_i^jf_k(x)f_k(y)-\sum_{h=1}^m\cos\angle a_i^jf_k(x)a_h^k\cdot\cos\angle a_h^kf_k(x)f_k(y)\right|<\varkappa(\delta).\]
Secondly, since $|xf_k(x)|\ll\ell\delta$, we have $|\tilde\angle a_i^jxa_h^k-\tilde\angle a_i^jf_k(x)a_h^k|<\varkappa(\delta)$ for all $1\le h\le m$.
Furthermore, since $a_i^j$, $b_i^j$ are on the shortest paths $p_j\alpha_i^j$, $p_j\beta_i^j$ at distance $\ell\delta$ from $p_j$, where $\{(\alpha_i^j,\beta_i^j)\}_{i=1}^m$ is an $(m,\delta)$-strainer at $p_j$ with length $>\ell$, we have 
\[\left|\angle a_i^jxa_h^k-\angle a_i^jf_k(x)a_h^k\right|<\varkappa(\delta)\]
for all $1\le h\le m$ (use the property \eqref{eq:str} twice).
Finally, by Lemma \ref{lem:susp}, there exists a $\varkappa(\delta)$-Hausdorff approximation $\Phi_2:\Sigma_x\to S^m(\hat\Sigma_2)$ which sends $\{((a_h^k)'_x,(b_h^k)'_x)\}_{h=1}^m$ to an orthogonal $m$-frame.
Furthermore, by the previous inequality, $\Phi_2$ sends $(a_i^j)'_x$ into the $\varkappa(\delta)$-neighborhood of $\mathbb S^{m-1}$ in $S^m(\hat\Sigma_2)$ since $\Phi_1$ sends $(a_i^j)'_{f_k(x)}$ into the $\varkappa(\delta)$-neighborhood of $\mathbb S^{m-1}$ in $S^m(\hat\Sigma_1)$.
Therefore, we have
\[\left|\cos\angle a_i^jxy-\sum_{h=1}^m\cos\angle a_i^jxa_h^k\cdot\cos\angle a_h^kxy\right|<\varkappa(\delta).\]
Combining the three inequalities above with the inequality \eqref{eq:clm2'} for $j=k$, we obtain \eqref{eq:clm2'} for $j\neq k$.
\end{proof}

By Lemma \ref{lem:grad}, take $\varepsilon>0$ such that $|\nabla_x\dist_E|>\varepsilon$ and $\vol_{n-1}\Sigma_x>\varepsilon$ for any $x\in U_\varepsilon(E)\setminus E$ (note that the choice of $\varepsilon$ depends on the diameter of $E$ and the volume of a small neighborhood of $E$).
Let $\delta\ll\varepsilon$ and $0<\rho\ll r<\ell\delta^2$.
Note that $U_\rho(E)\subset\bigcup_{l=1}^NU_l\subset U_\varepsilon(E)$.
Define $f:=f_N|_{U_\rho(E)}$.

Let us verify the properties (1)--(3) in Theorem \ref{thm:gnhd}.
Since $\rho\ll r$, it suffices to prove (1) on each $3U_j$.
Since $\psi_j$ is a $\varkappa(\delta)$-almost isometry, we may consider $\psi_j\circ f$ instead of $f$.
Then, the $(1+\varkappa(\delta))$-Lipschitzness of $\psi_j\circ f$ follows from that of $\varphi_j$ and Claim \ref{clm:2}.
Moreover, the $(1-\varkappa(\delta))$-openness of $\psi_j\circ f$ follows from Claim \ref{clm:2} and Lemma \ref{lem:open} since $\varphi_j$ satisfies the assumption of Lemma \ref{lem:open} (see the proof of Proposition \ref{prop:open}).
(2) is obvious from the construction of $f$.
(3) is proved in the same way as Corollary \ref{cor:lnhd}.
Recall that $(\varphi_j,|E\cdot|)$ is noncritical on $(U_\rho(E)\setminus E)\cap 3U_j$ in the sense of \cite{Per1}.
Thus, Claim \ref{clm:2} means that $(\psi_j\circ f,|E\cdot|)$ is noncritical on $(U_\rho(E)\setminus E)\cap 3U_j$ in the generalized sense of \cite[\S5]{F2}.
Therefore, (3) follows from the modified version of the fibration theorem \cite[5.2]{F2} in the same way as Corollary \ref{cor:lnhd}.
\end{proof}

In particular, under the assumption of Theorem \ref{thm:gnhd}, one can construct a deformation retraction of $U_\varepsilon(E)$ onto $E$ by crushing the fibers of $f$, using their cone structure.
In Appendix \ref{sec:A}, we construct such a deformation retraction for a general compact extremal subset.

Theorem \ref{thm:gnhd} (and its proof) can be applied to collapsing sequences:

\begin{cor}\label{cor:gnhd}
Let $M$, $E$, $\delta>0$ and $\rho>0$ be as in Theorem \ref{thm:gnhd}.
Let $\hat M$ be an Alexandrov spaces sufficiently close to $M$ such that $\dim\hat M>\dim M$.
Let $\hat E$ be a (not necessarily extremal) subset of $\hat M$ sufficiently close to $E$.
Assume further that $\delta\ll\vol\Sigma_x$ for any $x\in U_\rho(\hat E)$.
Then, $U_\rho(\hat E)$ admits a fiber bundle structure over $E$.
\end{cor}

\begin{proof}
The proof is almost the same as that of Theorem \ref{thm:gnhd}.
The outline is as follows.
Since the map $f$ of Theorem \ref{thm:gnhd} is constructed out of distance functions, it can be lifted to $\hat M$.
More precisely, one can construct a map $\hat f:U_\rho(\hat E)\to E$ just by replacing $\varphi_j$ and $U_j$ in the proof of Theorem \ref{thm:gnhd} with their natural lifts $\hat\varphi_j$ and $\hat U_j$, respectively.
Then, Claims \ref{clm:1} and \ref{clm:2} can be proved for $\hat f$ in the same way as above.
Furthermore, by the lower semicontinuity of angles, $\hat\varphi_j$ is noncritical on $U_\rho(\hat E)\cap 3\hat U_j$ and $(\hat\varphi_j,|\hat E\cdot|)$ is noncritical on $(U_\rho(\hat E)\setminus U_{\rho/2}(\hat E))\cap 3\hat U_j$.
Thus, the fibration theorem \cite[5.2]{F2} (together with \cite[\S1 Complement to Theorem A]{Per1}) yields a bundle structure of $U_\rho(\hat E)$ which respects $|\hat E\cdot|$ outside $U_{\rho/2}(\hat E)$.
\end{proof}

\appendix

\section{}\label{sec:A}

In this appendix, we prove Theorem \ref{thm:def}.
The proof is independent of the other results in this paper: it does not use the regular point theory.

\begin{proof}[Proof of Theorem \ref{thm:def}]
Let $M$ be an $n$-dimensional Alexandrov space with curvature $\ge\kappa$ and $E\subset M$ an extremal subset with diameter $\le D$ (in the extrinsic metric).
Assume $\vol_n(U_1(E))\ge v>0$.
We denote by $c$ and $C$ various small and large positive constants, respectively, which depend only on $n$, $\kappa$, $D$ and $v$.

Let $\nu>0$ be small enough and take a $\nu$-discrete set $\{q_\alpha\}_{\alpha=1}^N$ in $U_1(E)\setminus E$ such that $N\ge c/\nu^n$.
Consider a function
\[h:=\frac1N\sum_{\alpha=1}^N\dist_{q_\alpha}.\]
Let $0<\varepsilon\ll\min_\alpha|q_\alpha E|$.
We use the gradient flow of $h$ to push $U_\varepsilon(E)$ into $E$.

\begin{clm}[cf.\ Lemma \ref{lem:grad}]\label{clm:A1}
For any $x\in U_\varepsilon(E)\setminus E$, we have
\[\dist_E'(\nabla_xh)<-c.\]
In particular, $\dist_E$ is monotonically decreasing along the gradient flow of $h$.
\end{clm}

\begin{proof}
We first show that $h'(\xi)>c$ for any $\xi\in E'_x$.
The proof is almost the same as that of Lemma \ref{lem:grad}.
Here is an outline.
A standard comparison argument shows that the number of $\alpha$ such that $|\dist_{q_\alpha}'(\xi)|<\nu$ is less than $C/\nu^{n-1}$.
On the other hand, the extremality of $E$ implies $\dist_{q_\alpha}'(\xi)>-\nu$ since $\varepsilon\ll\min_\alpha|q_\alpha E|$.
Therefore, taking an average on them, we obtain
\[h'(\xi)=\frac1N\sum_{\alpha=1}^N\dist_{q_\alpha}'(\xi)>c\]
if $\nu$ is small enough.

By the property of the gradient (\cite[1.3.2]{Pet2}) and the first variation formula, we have
\[h'(\xi)\le\langle\nabla_xh,\xi\rangle\le-\dist_E'(\nabla_xh).\]
Together with the above, this implies the desired inequality.
\end{proof}

Let $\Phi:U_\varepsilon(E)\times[0,\infty)\to U_\varepsilon(E)$ be the gradient flow of $h$.
Note that $\Phi$ is Lipschitz (\cite[2.1.4]{Pet2}).
We need to modify it so that points of $E$ are fixed.
For $x\in U_\varepsilon(E)$, let $T(x)$ denote the minimum time such that $\Phi(x,T(x))\in E$.
By Claim \ref{clm:A1}, $T(x)$ is uniformly bounded above.

\begin{clm}[cf.\ {\cite[3.4]{MY}}]\label{clm:A2}
The function $x\mapsto T(x)$ is Lipschitz.
\end{clm}

Note that \cite[3.4]{MY} cannot be applied because Claim \ref{clm:A1} does not hold on $E$ (clearly, $\dist_E'(\nabla h)=0$ on $E$ by the extremality).
The author is grateful to the referee for the simplification of the following proof.

\begin{proof}
Let $x,y\in U_\varepsilon(E)$.
We may assume $T(x)<T(y)$.
Set $x':=\Phi(x,T(x))$ and $y':=\Phi(y,T(x))$.
Then, $x'\in E$ and $|x'y'|\le L|xy|$ for some constant $L>0$ by the Lipschitz continuity of $\Phi$.
In particular, $|Ey'|\le L|xy|$.
Furthermore, Claim \ref{clm:A1} implies that $T(y')\le |Ey'|/c$.
Thus, we have 
\[T(y)=T(x)+T(y')\le T(x)+\frac Lc|xy|,\]
which completes the proof.
\end{proof}

Set $T:=\sup_{x\in U_\varepsilon(E)}T(x)$ and define $\Psi:U_\varepsilon(E)\times[0,T]\to U_\varepsilon(E)$ by
\[\Psi(x,t):=\Phi\left(x,\frac{T(x)}{T}t\right).\]
Then, by Claim \ref{clm:A2}, $\Psi$ is a Lipschitz deformation retraction of $U_\varepsilon(E)$ onto $E$.
\end{proof}


\begin{thebibliography}{BGP}

\bibitem[BBI]{BBI}
D. Burago, Y. Burago and S. Ivanov, A course in metric geometry, Grad. Stud. Math., 33, Amer. Math. Soc., Providence, RI, 2001.

\bibitem[BGP]{BGP}
Yu. Burago, M. Gromov and G. Perel'man, A.D. Alexandrov spaces with curvature bounded below, Uspekhi Mat. Nauk 47 (1992), no. 2(284), 3--51, 222; translation in Russian Math. Surveys 47 (1992), no. 2, 1--58.

\bibitem[F1]{F1}
T. Fujioka, Uniform boundedness on extremal subsets in Alexandrov spaces, preprint \href{https://arxiv.org/abs/1809.00603}{arXiv:1809.00603}, 2018.

\bibitem[F2]{F2}
T. Fujioka, A fibration theorem for collapsing sequences of Alexandrov spaces, to appear in J. Topol. Anal., \href{https://doi.org/10.1142/S179352532150028X}{DOI:10.1142/S179352532150028X}, 2021.

\bibitem[HS]{HS}
J. Harvey and C. Searle, Orientation and symmetries of Alexandrov spaces with applications in positive curvature, J. Geom. Anal. 27 (2017), no. 2, 1636--1666.

\bibitem[K]{K}
V. Kapovitch, Perelman's stability theorem, Surv. Differ. Geom., XI, 103--136, Int. Press, Somerville, MA, 2007.

\bibitem[LiN]{LiN}
N. Li and A. Naber, Quantitative estimates on the singular sets of Alexandrov spaces, Peking Math. J. 3 (2020), 203--234.

\bibitem[L]{L}
A. Lytchak, Open map theorem for metric spaces, Algebra i Analiz 17 (2005), no.3, 139--159; St. Petersburg Math. J. 17 (2006), no. 3, 477--491.

\bibitem[LyN]{LyN}
A. Lytchak and K. Nagano, Geodesically complete spaces with an upper curvature bound, Geom. Funct. Anal. 29 (2019), 295--342.

\bibitem[MY]{MY}
A. Mitsuishi and T. Yamaguchi, Good coverings of Alexandrov spaces, Trans. Amer. Math. Soc. 372 (2019), 8107-8130.

\bibitem[OS]{OS}
Y. Otsu and T. Shioya, The Riemannian structure of Alexandrov spaces, J. Differential Geom. 39 (1994), no. 3, 629--658.

\bibitem[Per1]{Per1}
G. Perelman, Alexandrov's spaces with curvatures bounded from below II, preprint, 1991.

\bibitem[Per2]{Per2}
G. Ya. Perel'man, Elements of Morse theory on Aleksandrov spaces, Algebra i Analiz 5 (1993), no. 1, 232--241; translation in St. Petersburg Math. J. 5 (1994), no. 1, 205--213.

\bibitem[PP1]{PP1}
G. Ya. Perel'man and A. M. Petrunin, Extremal subsets in Aleksandrov spaces and the generalized Liberman theorem, Algebra i Analiz 5 (1993), no. 1, 242--256; translation in St. Petersburg Math. J. 5 (1994), no. 1, 215--227.

\bibitem[PP2]{PP2}
G. Perelman and A. Petrunin, Quasigeodesics and gradient curves in Alexandrov spaces, preprint, 1994.

\bibitem[Pet1]{Pet1}
A. Petrunin, Applications of quasigeodesics and gradient curves, Comparison geometry, 203--219, Math. Sci. Res. Inst. Publ., 30, Cambridge Univ. Press, Cambridge, 1997.

\bibitem[Pet2]{Pet2}
A. Petrunin, Semiconcave functions in Alexandrov's geometry, Metric and comparison geometry, 137--201, Surv. Differ. Geom., 11, Int. Press, Somerville, MA, 2007.

\bibitem[S]{S}
T. Shioya, Mass of rays in Alexandrov spaces of nonnegative curvature, Comment. Math. Helv. 69 (1994), no. 2, 208--228.

\bibitem[WSS]{WSS}
Y. Wang, X. Su and H. Sun, A new proof of almost isometry theorem in Alexandrov geometry with curvature bounded below, Asian J. Math. 17 (2013), no. 4, 715--728.

\bibitem[Y1]{Y1}
T. Yamaguchi, A convergence theorem in the geometry of Alexandrov spaces, Actes de la table ronde de g\'eom\'etrie diff\'erentielle, 601--642, S\'emin. Congr., 1, Soc. Math. France, Paris, 1996.

\bibitem[Y2]{Y2}
T. Yamaguchi, Collapsing 4-manifolds under a lower curvature bound, preprint, \href{https://arxiv.org/abs/1205.0323}{arXiv:1205.0323}, 2012.

\end{thebibliography}
\end{document}